\documentclass[preprint,12pt]{amsart}

\usepackage{dsfont}
\usepackage[latin1]{inputenx}
\usepackage{amsfonts}
\usepackage{latexsym}
\usepackage{enumerate}
\usepackage{amsthm,amsmath,amssymb,mathrsfs,yhmath}
\usepackage{tikz}
\usepackage{fullpage}
\usepackage{wasysym}
\usetikzlibrary{positioning}
\usetikzlibrary{matrix}
\usepackage{float}
\usepackage{subfig}

\usepackage[pdftex,bookmarks=true]{hyperref}
\usepackage{url}

\newcommand{\TT}{\mathbb{T}}

\newcommand{\abs}[1]{\lvert #1\rvert}

\newcommand{\kq}{\mathbf{k}Q}
\newcommand{\inti}{\operatorname{int}(\TT)}
\newcommand{\hh}{\operatorname{HH}}
\newcommand{\hc}{\operatorname{HC}}

\newtheorem{thm}{Theorem} 
\newtheorem*{thm1}{Theorem} 
\newtheorem{lemma}[thm]{Lemma}   

\newtheorem*{corollary}{Corollary}

\theoremstyle{remark}

\theoremstyle{definition}

\newtheorem{exam}[thm]{Example}
\newtheorem{rem}[thm]{Remark}

\begin{document}
\title{A note on (co)homologies of algebras from unpunctured surfaces}
\author{Yadira Valdivieso-D\'iaz}
 \email{yadira@matem.unam.mx}
\address{Instituto de Matem\'aticas, UNAM
 \'Area de la Investigaci\'on Cient\'ifica, Circuito exterior,
  Ciudad Universitaria, CDMX, 04510, M\'exico.} 

\begin{abstract}
In a previous paper, the author compute the dimension of Hochschild cohomology groups of Jacobian algebras from (unpunctured) triangulated surfaces, and gave a geometric interpretation of those numbers in terms of the number of internal triangles, the number of vertices and the existence of certain kind of boundaries. The aim of this note is computing the cyclic (co)homology and the Hochschild homology of the same family of algebras and giving an interpretation of those dimensions through elements of the triangulated surface.
\end{abstract}

\keywords{Hochschild homology, cyclic cohomology, cyclic homology, surfaces algebras}

\maketitle
\section{Introduction}
A \emph{surface with marked points}, or simply a \emph{surface}, is a pair $(S,M)$, where $S$ is a compact connected Riemann surface with (possibly empty) boundary, and $M$ is a non-empty finite subset of $S$ containing at least one point from each connected component of the boundary of $S$. We said that $(S,M)$ is an \emph{unpunctured surface} if $M$ is contained in the boundary of $S$. We define a \emph{triangulation} as a maximal collection of non-crossing arcs with endpoints in $M$. 

Given an \emph{(tagged) triangulation} $\TT$, it is possible to construct a finite dimensional algebra $A_\TT$, which turns out to be gentle if $(S,M)$ is an unpunctured surface (see \cite{LF09, ABCJP10, Lad12, TV15}), and in particular quadratic monomial. In this note, we compute  three different (co)homologies of those gentle algebras coming from unpunctured surfaces, namely: \emph{Hochschild homology} and \emph{cyclic homology and cohomology}, and we show that there is a combinatorial interpretation of those (co)homologies through the elements of the surface and the triangulation.


Given an associative algebra $A$ over a field $\mathbf k$ and $M$ an $A$-bimodule, we define 
the \emph{Hochschild cohomology} of $A$, with coefficients in $M$, as the graded vector space $\hh^*(A,M)= \operatorname{Ext}^*_{A\otimes_{\mathbf k} A^{\operatorname{op}}}(A,M)$, where $A^{\operatorname{op}}$ is the algebra $A$ with the opposite multiplication. The original definition was introduced by Hochschild in \cite{Hoc45} using a resolution of $A$ as bimodule. Later, Cartan and Eilenberg  in \cite[Chapter 9]{CE56} extended it to algebras over more general rings, and also dualized it, giving the definition of \emph{Hochschild homology}.



The \emph{cyclic cohomology} can be defined in several ways, but the original definition was given by Connes in \cite{Con81}, as a variation of the de Rham homology in spaces with bad behaviour. He used a sub-complex of the Hochschild complex when $M=\operatorname{Hom}_{\mathbf k}(A, {\mathbf k})$, called \emph{cyclic complex}.  In this note, we use the definition of cyclic (co)homology given in \cite[Theorem 4.1.13]{Lod92}, for algebras over rings of characteristic zero. As in the Hochschild cohomology case, the dual definition of cyclic cohomology was given later by several authors: Loday, Kassel, Quillen and Tsygan.


To compute those (co)homologies, we use the computation of the Hochschild homology of quadratic monomial algebras given by Sk\"olderberg in \cite{Sko99}, and two results of Loday \cite[Theorem 4.1.13, Section 2.4.8]{Lod92},  which relate the Hochschild homology and the cyclic homology by one hand, and the last one with the cyclic cohomology by other hand.

The main result of this note is the following:

\begin{thm1}
Let $(S,M, \TT)$ be a triangulated surface and $A_\TT$ be the algebra associated to $(S,M,\TT)$.  Then,
\begin{enumerate}
\item[i)] $\hh_n(A_\TT)\simeq \begin{cases}
\mathbf{k}(Q_\TT)_0  & \textrm{if $n=0$}\\
(A_\TT^!/[A_\TT^!, A_\TT^!])_{n} & \textrm{if $n\equiv 3\pmod{6}$}\\
(A_\TT^!/[A_\TT^!, A_\TT^!])_{n+1} & \textrm{if $n\equiv 2\pmod{6}$}\\
0 & \textrm{otherwise}
\end{cases}$

\item[ii)] $\hc_{n}(A_\TT)\simeq \begin{cases}
\mathbf{k} (Q_\TT)_0  & \textrm{if $n=0$}\\
\mathbf k (Q_\TT)_0  \oplus (A_\TT^!/[A_\TT^!, A_\TT^!])_{n+1} & \textrm{if $n\equiv 2\pmod{6}$}\\
0 & \textrm{otherwise}
\end{cases}
$

\end{enumerate}

Moreover, the dimension of the quotients $(A_\TT^!/[A_\TT^!, A_\TT^!])_{k}$ appearing in $\hh_n(A_\TT)$ and $\hc_n(A_\TT)$, are equal to the number of internal triangles of $\TT$.
\end{thm1}


As consequence of the main Theorem, it follows that the dimension of the cyclic cohomology is computed as follows.

\begin{corollary}
Let $A_\TT$ be the algebra associated to the triangulated surface $(S, M, \TT)$ and denote by $\inti$ the set of internal triangles of $\TT$. Then  $$\operatorname{dim}_{\mathbf k}(\hc^{n}(A_\TT))= \begin{cases}
\mid (Q_\TT)_0\mid  & \textrm{if $n=0$}\\
\mid (Q_\TT)_0\mid + \mid \inti \mid & \textrm{if $n\equiv 2\pmod{6}$}\\
0 & \textrm{otherwise}
\end{cases}
$$
\end{corollary}

According to the previous results, the dimensions of the cyclic (co)homology and Hochschild homology of the algebra $A_\TT$ depend on $\mid\inti\mid $ and the number of vertices of the quiver $Q_\TT$, then it is easy to observe those (co)homologies are not invariant under \emph{flips of arcs}, and therefore they are not invariant  under \emph{mutation of quivers with potentials}. See \cite{FST08} for definition of flips of arcs in triangulations, \cite{DWZ08} for definitions of mutations of quivers with potentials and \cite{LF09} for details about the relation between algebras from surfaces and quivers with potentials and its mutations.


\section{Results}
In the first part of this section, we recall some definitions and notations of path algebras and surfaces with marked points and we include the computation of the Hochschild homology given by Sk\"oldberg in \cite{Sko99} for completeness.

Let $Q=(Q_0, Q_1)$ be a finite quiver with a set of vertices $Q_0$ and a set of arrows$Q_1$. We denote the source and the target of an arrow $a\in Q_1$ by $s(a)$ and $t(a)$, respectively. A \emph{path} $w$ of length $l$ is a sequence of $l$ arrows $a_1\cdots a_l$, such that $t(a_k)=s(a_{k+1})$ for every $k=1, \dots, l-1$, we say that its source $s(w)$ is $s(\alpha_1)$ and its target $t(w)$ is $t(\alpha_l)$. We denote by $\abs{w}$ the length of the path $w$. We write $[e_i\mid a_1\cdots a_l\mid e_j]$ instead of $a_1 \cdots a_l$ to emphasize that the source of the path $a_1 \cdots a_l$ is  $e_i$ and its target is $e_j$.

Let $\mathbf k$ be an algebraically closed field. The path algebra $kQ$ is the $\mathbf k$-vector space with basis the set of paths in $Q$ and the product of the basis elements is given by the concatenations of the sequences of arrows of the paths $w$ and $w'$ if they form a path and zero otherwise. Let $F$ be the two-sided ideal of $kQ$ generated by the arrows of $Q$. A two-sided ideal $I$ is said to be admissible if there exists an integer $m_0\geq 2$ such that $F^{m_0}\subseteq I \subseteq F^2$ and its elements are called \emph{relations}. The pair $(Q,I)$ is called a \emph{bounded quiver}.

The quotient algebra $kQ/I$ is said a \emph{quadratic monomial algebra} if the admissible ideal $I$ is generated by paths of length 2. We associate its Koszul dual $A^!$ which is the quotient algebra $A^!=\kq/J$ where $J$ is generated by all paths $w$ of length 2 such that $w\notin I$.

In the next few paragraphs, we give a construction of a quadratic monomial algebra from an unpunctured surface.

Let $(S,M)$ be an unpunctured surface. An \emph{arc} $\tau$ in $(S,M)$ is a not self-crossing curve in $S$ with endpoints in $M$ and not isotopic to a point or to a boundary segment.

For any two arcs $\tau$ and $\tau'$ in $S$, let $e(\tau, \tau')$ be the minimal number of crossings of $\tau$ and $\tau'$, that is, $e(\tau, \tau')$ is the minimum of numbers of crossings of curves $\sigma$ and $\sigma'$, where $\sigma$ is isotopic to $\tau$ and $\sigma$ is isotopic to $\tau'$. Two arcs $\tau$ and $\tau'$ are called \emph{non-crossing} if $e(\tau,\tau')=0$. A \emph{triangulation} $\TT$ is a maximal collection of non-crossing arcs. The arcs of a triangulation $\TT$ cut the surface into \emph{triangles}. A triangle $\triangle$ in $\TT$ is called an \emph{internal triangle} if none of its sides is a boundary segment. We refer to the triple $(S,M,\TT)$ as a \emph{triangulated surface}.

If $\TT=\{\tau_1, \cdots\tau_m\}$ is a triangulation of an unpunctured surface $(S,M)$, we define a quiver $Q_\TT$ as follows: $Q_\TT$ has $m$ vertices, one for each arc in $\TT$. We will denote the vertex corresponding to $\tau_i$ by $e_i$ (or $i$ if there is no ambiguity). The number of arrows from $i$ to $j$ is the number of triangles $\triangle$ in $\TT$ such that the arcs $\tau_i, \tau_j$ form two sides of $\triangle$, with $\tau_j$ following $\tau_i$ when going around the triangle $\triangle$ in the counter-clockwise orientation. Note that the interior triangles in $\TT$ correspond to oriented 3-cycles in $Q_{\TT}$.

Following \cite{ABCJP10,LF09}, in the unpunctured case, the algebra $A_\TT$ is the quotient of the path algebra of the quiver $Q_\TT$ by the two-sided ideal generated by the subpaths of length two of each oriented 3-cycle in $Q_\TT$, then $A_\TT$ is a quadratic monomial algebra. It is easy to see that $A_\TT$ is also a gentle algebra.

Since $A_\TT$ is a quadratic monomial algebra, for any triangulated surface $(S,M,\TT)$, we use the computations of Sk\"oldberg in \cite[Corollary 1]{Sko99}to compute the Hochschild homology of $A_\TT$. Then, as we mention before, we use \cite[Theorem 4.1.13 and Section 2.4.8]{Lod92} to compute the dimension of the cyclic homology and cohomology. Before give the result of Sk\"oldberg we need to introduce some definitions and notations.

Let $A$ be a quadratic monomial algebra, observe that the algebra $A= \amalg_{n\in \mathbb N} A_n$ ( $A^!=\amalg_{n\in \mathbb N}A^!_{n})$ is $\mathbb N$-graded, where $A_n$ ($A_n^!$) is the $\mathbf k$ vector space generated by all paths of length $n$ of $A$(and $A^!$ respectively). This $\mathbb N$-graded is called the \emph{internal grading}. For a $\mathbb N$-graded vector space $V$, we define $V_{\geq i}$ by $\amalg_{i\geq n}V_i$ for a natural $n$.

We give both $A$ and its dual $A^!$ an other  $\mathbb N$-graded, called \emph{homological grading}, by assigning to a basis element $a_1 \cdots a_l$ of $A$ homological degree $0$, and to a basis element $b_1\cdots b_n$ of $A^!$ homological degree $n$. 

For two elements $x, y$ homogeneous with respect to the homological grading, in a bi-graded algebra, we define their graded commutator $[x, y]=xy-(-1)^{\operatorname{homdeg}(x)\operatorname{homdeg}(y)}yx$, where $\operatorname{homdeg}(y)$ is the homological degree, and we extend this definition bi-linearly to any pair of elements. If $A$ is a bi-graded algebra we define the vector space of commutators $[A,A]$ as

$$[A,A]=\operatorname{span}_{\mathbf k}\{[x,y]\mid x, y\in A\}.$$

\begin{rem}
Denote by $\mathcal C_m$ the set of cycles $Q$ of length $m$. Observe that the cyclic group $C_n=\langle g \rangle$ acts on $\mathcal C_n$ by the action $g a_1 \cdots a_n=a_n a_1 \cdots a_{n-1}$. We say that  two cycles $w$ and $w'$  are \emph{cyclically equivalent} if $w$ and $w'$ are in the same orbit.  Moreover, we say that any two cycles $w, w'$ are $[A,A]$-equivalent if $w$ and $w'$ are not elements of $[A,A]$ and they are cyclically equivalent, we denote this situation in the follow way $w\equiv w' \pmod{[A,A]}$.
\end{rem}

%
%
%
%
%
%
According to Sk\"oldberg \cite[Corollary 1]{Sko99} the Hochschild homology of a quadratic monomial algebra is computed as follows.
 
\begin{thm}\label{computations}
The Hochschild homology of any quadratic monomial algebra $A= \kq/I$ is given by

$$HH_{n}(A)\simeq \begin{cases}
\mathbf k Q_0 \oplus (A/[A,A])_{\geq 1} & \textrm{if $n=0$}\\

(A^!/[A^!, A^!])_{2} \oplus (A/[A,A])_{\geq 1} & \textrm{if $n=1$}\\
 (A^!/[A^!, A^!])_{n} \oplus (A^!/[A^!, A^!])_{n+1} & \textrm{if $n\geq 2$}
\end{cases}
$$

%
%
%

\end{thm}

The computation of the cyclic homology of a quadratic monomial algebra over a field, not necessarily of characteristic zero, was given by Sk\"oldberg in \cite{Sko01}. However, as consequence of Theorem \ref{computations} and the short exact sequence
 
$$0\longrightarrow \widetilde{\hc}_{n-1}\longrightarrow \widetilde{\hh}_{n}\longrightarrow \widetilde{\hc}_{n}\longrightarrow 0,$$
where $\widetilde{\hh}_{n}(A)$ is the quotient $\hh_n(A)/\hh_n(A_0)$ and $\widetilde{\hc}_{n}(A)$ is the quotient $\hc_n(A)/\hc_n(A_0)$,
see \cite[Theorem 4.1.13]{Lod92}, it is possible to compute the cyclic homology of a quadratic monomial algebra using an inductive argument, as follows.

\begin{thm}\cite[Theorem 4]{Sko01}\label{computations2}
Let $A$ be a quadratic monomial algebra, where $\mathbf k$ is a field of characteristic $0$, the cyclic homology of $A$ is given by
$$HC_{n}(A)\simeq \begin{cases}
\mathbf{k} Q_0 \oplus (A/[A,A])_{\geq 1} & \textrm{if $n=0$}\\
\mathbf k Q_0  \oplus (A^!/[A^!, A^!])_{n+1} & \textrm{if $n\equiv 0\pmod{2}$ and $n\neq 0$}\\
(A^!/[A^!, A^!])_{n+1} & \textrm{otherwise}
\end{cases}
$$

\end{thm}

As shown in Theorem \ref{computations} and  Theorem \ref{computations2}, to compute the cyclic and Hochschild homology of any algebra $A$, we need to compute the groups $A/[A,A]$ and $A^!/[A^!, A^!]$. In the following lemmas we show that those groups, for algebras coming from triangulated surfaces, are related to the internal triangles of the triangulation. In order to do that, we need to introduce some notation. Denote by $\operatorname{Int}(\TT)=\{\Delta_1, \cdots, \Delta_t\}$ the set of internal triangles of the triangulated surface $(S,M, \TT)$ and by $Q(\Delta_i)$ the subquiver of $Q_\TT$ associated to the internal triangle $\Delta_i$ for each $i=1, \dots, t$. By construction of $Q_\TT$, the quiver $Q(\Delta_i)$ is a 3-cycle.


\begin{rem}\label{paths}
By definition $[A, A]$ is generated by the elements $[x,y]$ such that $x, y\in A$. In particular, if $w=[e_i\mid a_1a_2\dots a_l\mid e_j]$ is a path such that $e_i\neq e_j$, we have that  $w= [w, e_j]$ . Then any path which is not a cycle is an element of $[A,A]$.
\end{rem}

\begin{lemma}\label{quotient1}
Let $(S,M, \TT)$ be a triangulated surface and $A_\TT$ be the algebra associated to $(S,M, \TT)$. Then the quotient $(A_\TT/[A_\TT,A_\TT])_{\geq 1}$ in trivial.
\end{lemma}

\begin{proof}
Since any path which is not a cycle is an element of $[A_{\TT},A_\TT]$, it is enough to show that any cycle , non-zero  in $A_\TT$ of positive length, is an element of the commutator $[A_\TT,A_\TT]$.

Suppose $w=[e_i\mid a_1a_2\dots a_l\mid e_i]$ is a cycle. Since $Q_\TT$ has no loops, we have that $\mid w \mid\geq 2$. Moreover, the algebra  $A_\TT$ is a finite dimensional quadratic monomial algebra, then $a_l a_i$ is an element of the ideal $I_\TT$, therefore $w=[a_1\cdots a_{l-1}, a_l]$, as we claim.
\end{proof}

\begin{lemma}\label{quotient2}
Let $(S, M, \TT)$ be a triangulated surface and $A_\TT$ be the algebra associated to $(S,M, \TT)$ and $n\neq 3(2t+1)$ for any $t\in\mathbb Z^{\geq 0}$. Then the quotient $(A_\TT^!/[A_\TT^!, A_\TT^!])_n$ is trivial. Moreover, if $n=3(2t+1)$ for some $t\in\mathbb Z^{\geq 0}$, then  $\operatorname{dim}_{\mathbf k}(A_\TT^!/[A_\TT^!, A_\TT^!])_n=\mid int(\TT)\mid$
\end{lemma}

\begin{proof}
By definition $A^!$ is the path algebra $kQ_\TT/J$, where $J$ is the ideal generated by the paths $m$ of length 2 such that $m\notin I$. Before give a basis for $A^!/[A^!, A^!]$, we first give the generators of $J$ and the non-zero paths in $A^!$.

Recall any path $w=[e_i\mid a_1a_2\dots a_l\mid e_j]$, non-zero in  $A$, is coming from arcs attached to a marked point $x$ as in Figure \ref{path}, and each arrow is opposite to $x$. Denote by $w_x$ the maximal non-zero path coming from the arcs attached to the marked point $x$. Then $J$ is generated by the subpaths of length $2$ of each maximals non-zero paths, therefore element of the basis of  $A^!$ is a sequence of consecutive arrows of a $3$-cycle $Q(\Delta)$ associated to an internal triangle $\Delta$ of $\TT$ or an arrow in $Q_\TT$.

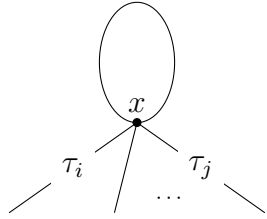
\begin{figure}[ht!]
\centering
\begin{tikzpicture}
\draw (0,0) ellipse (0.5cm and 0.8cm);
\filldraw [black] (0,-0.8) node[above]{$x$}circle (1.5pt);
\tikzset{every node/.style={fill=white}};
\draw (0,-0.8) to node[midway]{$\tau_j$} (1.7,-2)
		(0,-0.8) to (-0.3,-2)
		(0,-0.8) to node{$\tau_i$} (-1.7,-2);
\node at (0.45, -1.8){\tiny{$\cdots$}};
\end{tikzpicture}
\caption{Non-zero path}
\label{path}
\end{figure}

Since the length of any cycle in $A^!$ is always a multiple of $3$ and any path which is not a cycle is element of $[A^!, A^!]$, it is clear that $[A^!, A^!]\bigcap A^!_n=A^!_n$ for any $n$ which is not multiple of 3. 
 


Now, suppose $n$ is multiple of 3 and even. We claim that $[A^!, A^!]\bigcap A^!_n=A^!_n$. Let $w$ be path of length $n$. If $w$ is not a cycle, then $w$ is an element of $[A^!, A^!]\bigcap A_n$, see Remark \ref{paths}. Suppose $w$ is a cycle, then  $w=(a_1a_2a_3)^t$ for some $t\in\mathbb N$ such that $n=3t$. Therefore $[a_1a_2a_3, (a_1a_2a_3)^{t-1}]=(a_1a_2a_3)^t-(-1)^{(3)(3t-3)}(a_1a_2a_3)^t$, observe that $(-1)^{(3)(3t-3)}=-1$, then  $[a_1a_2a_3, (a_1a_2a_3)^{t-1}]=2(a_1a_2a_3)^t$, hence $[A^!, A^!]\bigcap A^!_n=A^!_n$.

Then, in both cases, when $n$ is not multiple of $3$ or $n$ is multiple of $3$ and even, we have that $(A^!/[A^!, A^!])_n=0$.

Finally, suppose $n$ is multiple of 3 and odd. By definition $$[w_1, w_2]=w_1w_2 - (-1)^{\mid w_1\mid \mid w_2\mid}w_2w_1$$  is a generator of $[A^!, A^!]\bigcap A^!_n$, if  $w_1$ and $w_2$ are paths such that the length $\mid w_1w_2\mid=n$ or $\mid w_2w_1\mid =n$. Moreover, since $n$ is odd and $n= \mid w_1 \mid +\mid w_2 \mid$, we have that $(-1)^{\mid w_1\mid \mid w_2\mid}=1$, then $$[w_1, w_2]=w_1w_2 - w_2w_1.$$

Let $g$ be a generator of the cyclic group $C_n$. We claim that any generator of $[A^!, A^!]\bigcap A^!_n$ is an element of the form:

\begin{itemize}
\item[(i)] $(c_1c_2c-3)^t -g^r(c_1c_2c_3)^t$ where $c_1c_2c_3$ is a 3-cycle coming from an internal triangle $\Delta_k$ of $\TT$  and $r\in \{1, 2, \dots, n-1\}$ or

\item[(ii)] a path $w'$ which is not a cycle, of length $n$.
\end{itemize}

Let $w_1=a_1a_2\dots a_{l_1}$ and $w_2=b_1b_2\dots b_{l_2}$ paths such that the length $\mid w_1w_2\mid=n$ or $\mid w_2w_1\mid=n$. Suppose $[w_1,w_2]$ is not a path, then $w_1w_2$ and $w_2w_1$ are non-zero cycles in $A^!$, and therefore both of them  are sequences of consecutive arrows of the same $3$-cycle $Q(\Delta)$ associated to an internal triangle $\Delta$ of $\TT$, hence $w_2w_1= g^r w_1w_2$ for some $r\in \{1,2\dots, n-1\}$ and $w_1w_2=(c_1c_2c_3)^t$, for some $c_1c_2c_3$ is a 3-cycle coming from an internal triangle $\Delta_k$ of $\TT$, as we claim.

And, as consequence, $\operatorname{dim}_{\mathbf k}((A^!/[A^!, A^!])_n)= \mid \inti\mid$.

\end{proof}

\begin{proof}[Proof of Main Theorem]
Let $A_\TT$ be the algebra associated to the triangulated surface $(S,M,\TT)$.

We first compute the cyclic homology groups $\hc_n(A_\TT)$. Since $(A_\TT/[A_\TT,A_\TT])_{\geq 1}$ is trivial by Lemma \ref{quotient1}, we have that $\hc_{0}(A_\TT)=\mathbf{k}Q_0$. Let $n\geq 1$, by Lemma \ref{quotient2} we have that the quotient $(A_\TT^!/[A_\TT^!, A_\TT^!])_{n+1}$ is also trivial for any $n+1\neq 3(2k+1)$, then by Theorem \ref{computations2} we have that the $n$-cyclic homology group of $A_\TT$ is given by
$$\hc_{n}(A_\TT)\simeq \begin{cases}
\mathbf k Q_0  & \textrm{if $n=0$}\\
\mathbf k Q_0  \oplus (A_\TT^!/[A_\TT^!, A_\TT^!])_{n+1} & \textrm{if $n\equiv 2\pmod{6}$}\\
0 & \textrm{otherwise}
\end{cases}
$$

Similarly by Theorem \ref{computations} and also Lemmas \ref{quotient1} and \ref{quotient2}, we have that the Hochschild homology groups of $A_\TT$ is given by 

$$\hh_n(A_\TT)\simeq \begin{cases}
\mathbf{k}(Q_\TT)_0  & \textrm{if $n=0$}\\
(A_\TT^!/[A_\TT^!, A_\TT^!])_{n} & \textrm{if $n\equiv 3\pmod{6}$}\\
(A_\TT^!/[A_\TT^!, A_\TT^!])_{n+1} & \textrm{if $n\equiv 2\pmod{6}$}\\
0 & \textrm{otherwise}
\end{cases}$$
\end{proof}

It follows from  Loday \cite[Section 2.4.8]{Lod92}, that  for any finite dimensional algebra $A$ with unit the $n$-cyclic homology group $\hc^n(A)$ of $A$ and the $n$-cyclic cohomology group $\operatorname{Hom}_k(\hc_n(A), k)$  of $A$ are isomorphic. Since $A_\TT$ is a finite dimensional algebra for any  triangulated surface $(S,M,\TT)$, the Corollary follows from the main Theorem.

To conclude this note, we compute the dimension of the cyclic (co)homologies groups and dimension of the Hochschild homology groups of two algebras from surfaces, which are closed related.

\begin{exam}
Consider the triangulated surface $(S,M,\TT)$ of the Figure \ref{3boundaries}, where $S$ is a surface of genus zero with 3 boundaries components and four marked points.

%

\begin{figure}[ht!]
\centering
\subfloat{
\begin{tikzpicture}[scale=.95]
\draw[thick](-0.2,1.5) circle(0.3cm)
(-.7,-.5) circle(0.3cm)
(3,1.8)circle(0.3cm);
\node at(-0.2,1.5){\tiny{$B_2$}};
\node at(-.7,-.5) {\tiny{$B_3$}};
\node at(3,1.8){\tiny{$B_1$}};
\draw (-.9,-.5) circle(0.5cm);
\tikzset{every node/.style={fill=white}};
\draw[black] (-.5,1.5) .. controls +(85:1.5cm) and +(180:1cm).. node[midway]{$\tau_3$}(3,1.5);
\draw[black] (-.5,1.5) .. controls +(-85:1.5cm) and +(-180:1cm).. node[midway]{$\tau_4$}(3,1.5);
\draw[black] (-.5,1.5) .. controls +(230:1.5cm) and +(170:1cm).. node[midway]{$\tau_1$}(-1.5,-1.3);
\draw(-1.5,-1.3) to [out=-10, in=-50, looseness=1.5] (-.4,-.5);
\draw[black] (-.5,1.5) .. controls +(100:3.5cm) and +(10:3cm).. node[midway]{$\tau_2$}(3,1.5);
\draw(-.5,1.5) to  node[midway]{$\tau_5$}(-.4,-.5);
\draw(-.4, -.5) to node[midway]{$\tau_6$}(3,1.5);
\node at (-1.2, -1){$\tau_7$};
\filldraw [black] (-.5,1.5) circle (1.5pt)
					(3,1.5) circle (1.5pt)
					(-.4,-.5)circle(1.5pt)
					(-1,-.5)circle(1.5pt);
\end{tikzpicture}
}
 \subfloat{
  \begin{tikzpicture}[scale=.6]
 \matrix (m)[matrix of math nodes, row sep=1.5em,column sep=3em,ampersand replacement=\&]
{7  \&  \&  4\&  \\
    \& 5\&   \&  \\
  1 \&  \&  6\& 3\\
    \& 2\&   \&  \\
};
\path[-stealth]
	(m-1-1) edge (m-3-1)
	(m-3-1) edge (m-2-2) edge (m-3-3)
	(m-2-2) edge(m-1-1) edge (m-1-3)
	(m-3-3) edge(m-2-2) edge(m-4-2)
	(m-4-2)edge(m-3-1)
	(m-1-3)edge(m-3-3)
	(m-3-4) edge (m-1-3) edge(m-4-2);
\end{tikzpicture}
}
\caption{The triangulated surface $(S,M,\TT)$}
\label{3boundaries}
\end{figure}

Observe that the quiver $Q_\TT$ has 7 vertices and there are 3 internal triangles: $\triangle_1(\tau_1, \tau_5, \tau_7)$, $\triangle_2(\tau_1,\tau_6,\tau_2)$ and $\triangle_3(\tau_6,\tau_5,\tau_4)$. Then, according to our main Result, the dimension of the Hochschild homology groups $\hh_n(A_\TT)$ are computed as follows:

 $$\operatorname{dim}_{\mathbf k}\hh_n(A_\TT)= \begin{cases}
7  & \textrm{if $n=0$}\\
3 & \textrm{if $n\equiv 3\pmod{6}$ or $n\equiv 2\pmod{6}$ }\\
0 & \textrm{otherwise}
\end{cases}$$

Since $\operatorname{dim}_{\mathbf k}(\hc^n(A_\TT))=\operatorname{dim}_{\mathbf k}(\hc_n(A_\TT))$  by Corollary for any $n\in \mathbb Z^{\geq 0}$, the dimension of the cyclic (co)homology groups $\hc_n(A_\TT)$ and $\hc^n(A_\TT)$ are computed as follows:

$$\operatorname{dim}_{\mathbf k}(\hc^n(A_\TT))=\operatorname{dim}_{\mathbf k}(\hc_n(A_\TT))= \begin{cases}
7  & \textrm{if $n=0$}\\
10 & \textrm{if $n\equiv 2\pmod{6}$}\\
0 & \textrm{otherwise}
\end{cases}
$$

Finally, consider the triangulated surface $(S,M, \TT')$ of Figure \ref{3boundaries2}, which is obtained by removing the arc $\tau_7$ of the triangulation $\TT$ and replacing it by $\tau_7'$, that is, a \emph{flip of arc $\tau_7$}. 

\begin{figure}[ht!]
\centering
\subfloat{
\begin{tikzpicture}[scale=.97]
\draw[thick](-0.2,1.5) circle(0.3cm)
(-.7,-.5) circle(0.3cm)
(3,1.8)circle(0.3cm);
\node at(-0.2,1.5){\tiny{$B_2$}};
\node at(-.7,-.5) {\tiny{$B_3$}};
\node at(3,1.8){\tiny{$B_1$}};
\tikzset{every node/.style={fill=white}};
\draw[black] (-.5,1.5) .. controls +(85:1.5cm) and +(180:1cm).. node[midway]{$\tau_3$}(3,1.5);
\draw[black] (-.5,1.5) .. controls +(-85:1.5cm) and +(-180:1cm).. node[midway]{$\tau_4$}(3,1.5);
\draw[black] (-.5,1.5) .. controls +(230:1.5cm) and +(170:1cm).. node[midway]{$\tau_1$}(-1.5,-1.3);
\draw(-1.5,-1.3) to [out=-10, in=-50, looseness=1.5] (-.4,-.5);
\draw[black] (-.5,1.5) .. controls +(100:3.5cm) and +(10:3cm).. node[midway]{$\tau_2$}(3,1.5);
\draw (-.5,1.5) to (-1,-.5);
\draw(-.5,1.5) to  node[midway]{$\tau_5$}(-.4,-.5);
\draw(-.4, -.5) to node[midway]{$\tau_6$}(3,1.5);
\node at (-.86, 0.26){$\tau_7'$};
\filldraw [black] (-.5,1.5) circle (1.5pt)
					(3,1.5) circle (1.5pt)
					(-.4,-.5)circle(1.5pt)
					(-1,-.5)circle(1.5pt);
\end{tikzpicture}
}
 \subfloat{
  \begin{tikzpicture}[scale=1.2]
 \matrix (m)[matrix of math nodes, row sep=1.5em,column sep=3em,ampersand replacement=\&]
{7  \&  \&  4\&  \\
    \& 5\&   \&  \\
  1 \&  \&  6\& 3\\
    \& 2\&   \&  \\
};
\path[-stealth]
	(m-1-1) edge (m-2-2)
	(m-3-1) edge (m-1-1)
	(m-3-1) edge (m-3-3)
	(m-2-2) edge (m-1-3)
	(m-3-3) edge(m-2-2) edge(m-4-2)
	(m-4-2) edge(m-3-1)
	(m-1-3) edge(m-3-3)
	(m-3-4) edge (m-1-3) edge(m-4-2);
\end{tikzpicture}
}
\caption{The triangulated surface $(S,M,\TT')$}
\label{3boundaries2}
\end{figure}

In this case, the quiver $Q_\TT'$ has also 7 vertices, which is actually an invariant of $(S,M)$, but there are only 2 internal triangles: $\Delta_2(\tau_1,\tau_6,\tau_2)$ and $\Delta_3(\tau_6, \tau_5, \tau_4)$ , then:

 $$\operatorname{dim}_{\mathbf k}\hh_n(A_\TT')= \begin{cases}
2 & \textrm{if $n\equiv 3\pmod{6}$ or $n\equiv 2\pmod{6}$ }\\
\operatorname{dim}_{\mathbf k}\hh_n(A_\TT) & \textrm{otherwise}
\end{cases}$$
and
$$\operatorname{dim}_{\mathbf k}(\hc^n(A_\TT'))=\operatorname{dim}_{\mathbf k}(\hc_n(A_\TT'))= \begin{cases}
9 & \textrm{if $n\equiv 2\pmod{6}$}\\
\operatorname{dim}_{\mathbf k}(\hc^n(A_\TT)) & \textrm{otherwise}
\end{cases}
$$

Therefore, the (co)homologies computed in this note are not invariant under \emph{flips} and \emph{mutations of quivers with potentials}.

\end{exam}



\begin{thebibliography}{10}

\bibitem{ABCJP10}
I.~Assem, T.~Br{\"u}stle, G.~Charbonneau-Jodoin, and P.-G. Plamondon.
\newblock Gentle algebras arising from surface triangulations.
\newblock {\em Algebra Number Theory}, 4(2):201--229, 2010.

\bibitem{CE56}
H.~Cartan and S.~Eilenberg.
\newblock {\em Homological algebra}.
\newblock Princeton University Press, Princeton, N. J., 1956.

\bibitem{Con81}
A.~Connes.
\newblock Spectral sequence and homology of currents for operator algebras.
\newblock {\em Math. Forschungsinst. Oberwolfach Tagungsber}, 41(81):27--9,
  1981.

\bibitem{DWZ08}
H.~Derksen, J.~Weyman, and A.~Zelevinsky.
\newblock Quivers with potentials and their representations. {I}. {M}utations.
\newblock {\em Selecta Math. (N.S.)}, 14(1):59--119, 2008.

\bibitem{FST08}
S.~Fomin, M.~Shapiro, and D.~Thurston.
\newblock Cluster algebras and triangulated surfaces. {I}. {C}luster complexes.
\newblock {\em Acta Math.}, 201(1):83--146, 2008.

\bibitem{Hoc45}
G.~Hochschild.
\newblock On the cohomology groups of an associative algebra.
\newblock {\em Ann. of Math. (2)}, 46:58--67, 1945.

\bibitem{LF09}
D.~Labardini-Fragoso.
\newblock Quivers with potentials associated to triangulated surfaces.
\newblock {\em Proc. Lond. Math. Soc. (3)}, 98(3):797--839, 2009.

\bibitem{Lad12}
S.~Ladkani.
\newblock On \textsc{J}acobian algebras from closed surfaces.
\newblock arXiv:1207.3778.

\bibitem{Lod92}
J.-L. Loday.
\newblock {\em Cyclic homology}, volume 301 of {\em Grundlehren der
  Mathematischen Wissenschaften [Fundamental Principles of Mathematical
  Sciences]}.
\newblock Springer-Verlag, Berlin, 1992.
\newblock Appendix E by Mar{\'{\i}}a O. Ronco.

\bibitem{Sko99}
E.~Sk{\"o}ldberg.
\newblock The {H}ochschild homology of truncated and quadratic monomial
  algebras.
\newblock {\em J. London Math. Soc. (2)}, 59(1):76--86, 1999.

\bibitem{Sko01}
E.~Sk{\"o}ldberg.
\newblock Cyclic homology of quadratic monomial algebras.
\newblock {\em J. Pure Appl. Algebra}, 156(2-3):345--356, 2001.

\bibitem{TV15}
S.~Trepode and Y.~Valdivieso-D{\'i}az.
\newblock On finite dimensional {J}acobian algebras.
\newblock {\em Bolet{\'i}n de la Sociedad Matem{\'a}tica Mexicana}, pages
  1--14, 2015.

\end{thebibliography}

\end{document}